\documentclass[11pt]{article}
\usepackage{amsmath}
  \usepackage{paralist}
  \usepackage{graphics}
  \usepackage{epsfig}
  \usepackage{float}
  \usepackage{graphicx}
  \usepackage{abstract}
 \usepackage{epstopdf}
 \usepackage{amssymb}
 \usepackage{cite}
 \usepackage{mathrsfs}
 \usepackage{amsthm}
 \usepackage[colorlinks=true]{hyperref}
\hypersetup{urlcolor=blue, citecolor=red}

\numberwithin{equation}{section}

\newtheorem{theorem}{Theorem}[section]

\newtheorem{proposition}{Proposition}

\newtheorem{physical conclusion}{Physical Conclusion}

\newtheorem{remark}{Remark}

\title{Dynamic bifurcation and instability of Dean problem}
\author{\scshape Huichao Wang \thanks{Email:huichao87@qq.com}
\\ \footnotesize Department of Mathematics
\\ \footnotesize Sichuan University
\\ \footnotesize Chengdu,
Sichuan 610064, China
\medskip
\\ \scshape   Quan Wang \thanks{Corresponding author:wqxihujunzi@126.com}
\\ \footnotesize Department of Mathematics
\\ \footnotesize Sichuan University
\\ \footnotesize Chengdu,
Sichuan 610064, China
\medskip
\\ \scshape   Ruikuan Liu\thanks{Email:liuruikuan2008@163.com. Supported by NSFC(11401479)}
\\ \footnotesize Department of Mathematics
\\ \footnotesize Sichuan University
\\ \footnotesize Chengdu,
Sichuan 610064, China }

\begin{document}
\date{}
\maketitle
\begin{abstract}
The main objective of this paper is to address the instability and dynamical bifurcation of the Dean problem. A nonlinear theory is obtained for the Dean problem, leading in particular to rigorous justifications of the linear theory used by physicists, and the vortex structure. The main technical tools are the dynamic bifurcation theory\cite{Ma} developed recently by Ma and Wang.
\begin{center}
\textbf{\normalsize keywords}
\end{center}
Dean problem; instability; dynamic bifurcation.
\end{abstract}

\section{Introduction}

The instability of rotating flows is an important issue in fluid dynamics. The  problem go back to the pioneering work of Rayleigh\cite{9}. He also considered a basic swirling flow of an inviscid fluid which moves with angular velocity $\Omega(r)$, an arbitrary function of the distance $r$ from the axis of rotation, which is determined
by the radial rate of change of pressure, then Rayleigh's criterion for stability is ``the
square of the circulation, $r^2\Omega^2$,  is an increasing function of.''  Synge has pointed out that Rayleigh's argument is inconclusive, and that a rigorous demonstration of stability can only be determined by considering the stability of the basic flow to arbitrary perturbations. Further, Synge showed that Rayleigh's criterion is valid for arbitrary perturbations which are symmetric with respect to the axis of rotation\cite{C}. However, Rayleigh's criterion can predicts instability for an inviscid rotating fluid. The stability of an inviscid fluid with respect to non-axisymmetric disturbance has also been studied by Bisshopp\cite{2}  and
Krueger and Diprima\cite{4}  in the narrow-gap approximation.

Note that the instability of viscous rotation flows also have been considered by several authors
including Taylor\cite{10} and Dean\cite{2}. In 1923, Taylor \cite{10} conducted a famous experiment, and observed and studied the stability of an incompressible viscous flow between two rotating coaxial cylinders. He found that when the Taylor number $T$ is smaller than a critical value $T_{c}>0$, called the critical Taylor number, the basic flow, called the Couette flow, is stable. While the Taylor number crosses the critical number, the Couette flow breaks into a cellular pattern that is radially symmetric. Such instability called Taylor instability. More research about Taylor instability see \cite{1,C,D,K,Y}. In 1928, Dean\cite{2} found that a similar type of instability can also occur when a viscous fluid in a curved channel owing to a pressure gradient acting round the channel. Such instability is called Dean instability, which has been considered again by H$\ddot{a}$mmerl\cite{3} and Walowit, Tsao and DiPrima \cite{11}.
For Couette flow, especially, in the case that the cylinders rotate in the same direction, a simple formula for predicting the critical speed is derived in \cite{11}.  The effect of a radial temperature gradient on the stability of Couette flow is also considered.

 Recently, Ma and Wang \cite{Ma0,Ma} have developed a bifurcation theory [9] for nonlinear partial differential equations, which has been used to develop a nonlinear analysis for the
 Rayleigh-B\'{e}nard convections \cite{Ma1,Ma2} and Taylor stability and bifurcation \cite{Ma3,Ma4}. This bifurcation theory is centered on a new notion
of bifurcation, called attractor bifurcation, for nonlinear evolution equations. In particular, based on this bifurcation theory, for the Rayleigh-B\'{e}nard convection
problem, Ma and Wang have shown \cite{Ma1,Ma2} that the problem bifurcates from the trivial solution to an attractor $\mathcal{A}_R$ when the Rayleigh number $R$ crosses the first critical Rayleigh number $R_c$ , for all physically sound boundary conditions and regardless of the multiplicity of the eigenvalue $R_c$ for the linear problem.
It is worth pointing out that the instability of viscous flows is the nonlinear phenomena.

Motivated by the above papers \cite{Ma1,Ma2,Ma3,Ma4}, our purpose of this article is to address the dynamic transition of the Dean problem under narrow gap approximation. The main
technical tool are the dynamical  bifurcation theory \cite{Ma0,Ma}. The basis of our research is a basic solution of Dean  problem as follows
\begin{equation*}
  u_{\theta}^{0}=\frac{1}{2\rho\nu}\bigg(\frac{\partial p}{\partial \theta}\bigg)_{0}(r\ln r+Ar+\frac{B}{r}),
\end{equation*}
where $A$ and $B$ are determined by boundary velocity. In this paper, we shall get the expressions of $A$ and $B$ by the physical assumption, which are different from expressions in \cite{2,3,11} where the boundary velocity vanishes.  Based on the basic solution determined by boundary condition and narrow gap approximation, we get a new simplified  governing model for
the Dean problem. In addition, the expressions of $A$ and $B$ guarantee that the linear part of
dynamical equation is symmetrical.

The article is organized as follows. In Section 2, we will give one type of boundary conditions for the steady state solution to guarantee that the corresponding  linear part is symmetrical.
The simplified model and the set-up are
given in Section 2, and the principle of exchange of stability(PES) is given in Section 3, and the main theorems are proved in Section 4.

\section{Simplified Governing equations}
\subsection{the basic flow}
The governing equations of a viscous incompressible fluid in the space between two concentric cylinders are the following equations in the cylindrical coordinates, which are given by
 \begin{eqnarray}
\label{21} \left\{
   \begin{aligned}
 &\frac{\partial u_{r}}{\partial t}+(u \cdot \nabla)u_{r}-\frac{u_{\theta}^{2}}{r}=\nu(\Delta u_{r}-\frac{u_{r}}{r^{2}}-\frac{2}{r^{2}}\frac{\partial u_{\theta}}{\partial \theta})-\frac{1}{\rho}\frac{\partial p}{\partial r},
 \\&\frac{\partial u_{\theta}}{\partial t}+(u \cdot \nabla)u_{\theta}+\frac{u_{r}u_{\theta}}{r}=\nu(\Delta u_{\theta}-\frac{u_{\theta}}{r^2}+\frac{2}{r^{2}}\frac{\partial u_{r}}{\partial \theta})-\frac{1}{r\rho}\frac{\partial p}{\partial \theta},
 \\& \frac{\partial u_{z}}{\partial t}+(u \cdot \nabla)u_{z}=\nu\Delta u_{z}-\frac{1}{\rho}\frac{\partial p}{\partial z},
 \\& \frac{\partial(ru_{r})}{\partial r}+\frac{\partial u_{\theta}}{\partial \theta}+\frac{\partial(ru_{z})}{\partial z}=0,
\end{aligned}
\right.
\end{eqnarray}
where $\nu$ is the kinematic viscosity, $\rho$ is the
density, $u=(u_{r},u_{\theta},u_{z})$ is the velocity field, $p$ is the pressure function, and
\begin{eqnarray*}
\begin{aligned}
  &(u \cdot \nabla)=u_{r}\frac{\partial}{\partial r}+\frac{u_{\theta}}{r}\frac{\partial}{\partial \theta}+u_{z}\frac{\partial}{\partial z},
  \\&\Delta=\frac{1}{r}\frac{\partial}{\partial r}\bigg(r\frac{\partial}{\partial r}\bigg)+\frac{1}{r^2}\frac{\partial^{2}}{\partial \theta^{2}}+\frac{\partial^{2}}{\partial z^{2}}.
\end{aligned}
\end{eqnarray*}

To simplify (\ref{21}), we make the following physical assumptions:
\begin{enumerate}
\item narrow gap approximation,
\begin{equation}\label{ng}
R_2-R_1\leq\frac{R_1+R_2}{2},
\end{equation}
where $R_1$ and $R_2$ are the radius of
of the two concentric cylinders.
\item the steady state solution of (\ref{21}) only related to the $\theta$-direction and
only depend on the variable $r$, i.e.,
\begin{equation}\label{bs}
u=(0,\overline{u}_\theta(r),0).
\end{equation}
\item the partial derivatives of pressure function $\widetilde{p}$  with respect to $\theta$ is
a constant, i.e.,
\begin{equation}\label{p}
\frac{\partial \overline{p}}{\partial \theta}=\bigg(\frac{\partial p}{\partial \theta}\bigg)_0=\text{const.}
\end{equation}
\end{enumerate}
From \cite{D}, the basic flow for (\ref{21}) is a steady state solution, defined by
\begin{equation}\label{bss}
\left\{
\begin{aligned}
& u_r=u_z=0,\quad u_\theta=\overline{u}_{\theta}(r)=\frac{1}{2\rho\nu}\bigg(\frac{\partial p}{\partial \theta}\bigg)_{0}(r\ln r+Ar+\frac{B}{r}),
\\ & \frac{1}{\rho}\frac{\partial \overline{p}}{\partial r}=\frac{u_{\theta}^{2}(r)}{r},\quad \frac{\partial \overline{p}}{\partial \theta}=\bigg(\frac{\partial p}{\partial \theta}\bigg)_{0}.
\end{aligned}
\right.
\end{equation}
\begin{remark}
In this paper, the basic flow (\ref{p}) we considered is different from the  Couette flow for
the Taylor problem.
\end{remark}
By the boundary conditions
\begin{equation}\label{bc}
\overline{u}_{\theta}(R_{1})=\frac{1}{2\rho\nu}\bigg(\frac{\partial p}{\partial \theta}\bigg)_{0}R_{2}^{2},\quad \overline{u}_{\theta}(R_{2})=\frac{1}{2\rho\nu}\bigg(\frac{\partial p}{\partial \theta}\bigg)_{0}R_{1}^{2}.
\end{equation}
We also assume that the pressure $p$ is a function depending on $\theta$ and $r$, satisfying that
\begin{equation}\label{p11}
\frac{1}{\rho}\frac{\partial p}{\partial r}=\frac{u_{\theta}^{2}(r)}{r},\quad \frac{\partial p}{\partial \theta}=\bigg(\frac{\partial p}{\partial \theta}\bigg)_{0}.
\end{equation}
Basing on the above assumptions, together (\ref{bss}) with (\ref{bc}), we get
\begin{align}
&\label{A} A=-\frac{R_{2}^{2}\ln R_{2}-R_{1}^{2}\ln R_{1}+R_{2}^{2}R_{1}-R_{1}^{2}R_{2}}{R_{2}^{2}-R_{1}^{2}},
\\& \label{B} B=\frac{R_{1}^{2}R_{2}^{2}(\ln R_{2}-\ln R_{1})+R_{2}^{4}R_{1}-R_{1}^{4}R_{2}}{R_{2}^{2}-R_{1}^{2}}.
\end{align}
\begin{remark}
Note that (\ref{A}), (\ref{B}) and the narrow gap approximation assumption (\ref{ng}) can
guarantee the linear part of the perturbation equation by (\ref{bss}) is symmetrical.
\end{remark}

\subsection{perturbed non-dimensionless equation}
In the following, in order to investigate the stability of the flow described by (\ref{bss}), we need to consider the perturbed state
\begin{equation*}
  u_r,~\overline{u}_\theta+u_\theta,~u_z,\quad p+\overline{p}.
\end{equation*}
Assume that the perturbations are axi-symmetric and independent of $\theta$, we derive from (\ref{21}) that
\begin{eqnarray}
\label{22} \left\{
   \begin{aligned}
 &\frac{\partial u_{r}}{\partial t}+(u \cdot \nabla)u_{r}-\frac{u_{\theta}^{2}}{r}=\nu\bigg(\Delta u_{r}-\frac{u_{r}}{r^{2}}\bigg)-\frac{1}{\rho}\frac{\partial p}{\partial r}
  \\& \qquad\qquad\qquad\qquad\qquad+\frac{1}{2\nu\rho}\bigg(\frac{\partial p}{\partial \theta}\bigg)_0\bigg(2\ln r +2A +\frac{2B}{r^{2}}\bigg)u_\theta,
 \\&\frac{\partial u_{\theta}}{\partial t}+(u \cdot \nabla)u_{\theta}+\frac{u_{r}u_{\theta}}{r}=\nu\bigg(\Delta u_{\theta}-\frac{u_{\theta}}{r^{2}}+\frac{1}{\nu\rho}\frac{1}{r}\bigg)
 \\& \qquad\qquad\qquad\qquad\qquad-\frac{1}{2\nu\rho}\bigg(\frac{\partial p}{\partial \theta}\bigg)_0(2\ln r +2A +1)u_{r},
 \\& \frac{\partial u_{z}}{\partial t}+(u \cdot \nabla)u_{z}=\nu\Delta u_{z}-\frac{1}{\rho}\frac{\partial p}{\partial z},
 \\& \frac{\partial}{\partial r}(ru_{r})+\frac{\partial}{\partial z}(ru_{z})=0,
\end{aligned}
\right.
\end{eqnarray}
where
\begin{eqnarray}
&\begin{aligned}
 \label{23} (u \cdot \nabla)= u_{r}\frac{\partial}{\partial r}+u_{z}\frac{\partial}{\partial z},
\end{aligned}
\\&\begin{aligned}
\label{24}  \Delta = \frac{\partial^{2}}{\partial r^{2}}+\frac{1}{r} \frac{\partial}{\partial r}+\frac{\partial^{2}}{\partial z^{2}}.
\end{aligned}
\end{eqnarray}
To derive the non-dimensionless form of equations (\ref{22}), let
\begin{eqnarray*}
\begin{aligned}
&(x,t)=\bigg(lx', \frac{l^{2}}{\nu}t'\bigg),\quad\quad\quad(x=(r,r\theta,z)),
\\&(u,p)=\bigg(\frac{\nu}{l}u', \frac{\rho\nu^{2}}{l^{2}}p'\bigg),\quad\quad(u=(u_r,u_\theta,u_z)).
\end{aligned}
\end{eqnarray*}
Omitting the primes, we obtain the non-dimensionless form of (\ref{22}) as follows
\begin{eqnarray}
\label{25} \left\{
   \begin{aligned}
 &\frac{\partial u_{r}}{\partial t}+(u \cdot \nabla)u_{r}-\frac{u_{\theta}^{2}}{r}=
 \Delta u_{r}-\frac{u_{r}}{r^{2}}-\frac{\partial p}{\partial r}
   +\bigg(\frac{\partial p}{\partial \theta}\bigg)_0\bigg(\ln r +A +\frac{B}{r^{2}}\bigg)u_{\theta},
 \\&\frac{\partial u_{\theta}}{\partial t}+(u \cdot \nabla)u_{\theta}+\frac{u_{r}u_{\theta}}{r}=\Delta u_{\theta}-\frac{u_{\theta}}{r^{2}}+\frac{2}{r}\bigg(\frac{\partial p}{\partial \theta}\bigg)_0
 \\& \qquad\qquad\qquad\qquad\qquad\qquad-\bigg(\frac{\partial p}{\partial \theta}\bigg)_0\bigg(\ln r +A +\frac{1}{2}\bigg)u_{r},
 \\& \frac{\partial u_{z}}{\partial t}+(u \cdot \nabla)u_{z}=\Delta u_{z}-\frac{\partial p}{\partial z},
 \\& \frac{\partial}{\partial r}(ru_{r})+\frac{\partial}{\partial z}(ru_{z})=0,
\end{aligned}
\right.
\end{eqnarray}

Taking the length scale $l=R_{2}-R_{1}$, then the narrow gap condition is given by
\begin{equation}\label{26}
  \frac{l}{R_{1}}\ll 1.
\end{equation}

Furthermore, by (\ref{26}) and (\ref{A}), we have
\begin{equation}\label{A1}
\begin{aligned}
&\ln r+A+\frac{1}{2}
\\&=\ln r-\frac{R_{2}^{2}\ln R_{2}-R_{1}^{2}\ln R_{1}+R_{2}^{2}R_{1}-R_{1}^{2}R_{2}}{R_{2}^{2}-R_{1}^{2}}+\frac{1}{2}
\\ &=\ln r-\ln R_2-\frac{R^2_1(\ln R_2-\ln R_1)}{R_{2}^{2}-R_{1}^{2}}
-\frac{R_1R_2}{R_1+R_2}+\frac{1}{2}
\\ &=\ln \frac{r}{R_2}-\frac{\ln (1+l/R_1)}{(1+l/R_1)^2-1}
-\frac{R_1R_2}{R_1+R_2}+\frac{1}{2}
\\ &\sim -\frac{l/R_1}{l/R_1(1+l/R_1)}-\frac{\ln (1+l/R_1)}{(1+l/R_1)^2-1}
-\frac{R_1R_2}{R_1+R_2}+\frac{1}{2}
\\ &\sim-\frac{R_1R_2}{R_1+R_2},
\end{aligned}
\end{equation}
Moreover,

\begin{equation}\label{B1}
\begin{aligned}
&\ln r+A+ \frac{B}{r^2}
\\&=\ln r+A+\frac{1}{2}+\frac{B}{r^2}-\frac{1}{2}
\\ &\sim -\frac{R_1R_2}{R_1+R_2}+\frac{1}{r^2}\frac{R_{1}^{2}R_{2}^{2}(\ln R_{2}-\ln R_{1})+R_{2}^{4}R_{1}-R_{1}^{4}R_{2}}{R_{2}^{2}-R_{1}^{2}}-\frac{1}{2}
\\ &=-\frac{R_1R_2}{R_1+R_2}+\frac{R^2_2}{r^2}
\frac{R_{1}^{2}(\ln R_{2}-\ln R_{1})}{R_{2}^{2}-R_{1}^{2}}+\frac{R_{1}R_{2}}{r^2}\frac{R^3_2-R^3_1}{R_{2}^{2}-R_{1}^{2}}
-\frac{1}{2}
\\ &\sim -\frac{R_1R_2}{R_1+R_2}+\frac{1}{2}
+\frac{R_{2}^{2}+R_1R_2+R_{1}^{2}}{R_{2}^{2}-R_{1}^{2}}-\frac{1}{2}
\\ &\sim\frac{2R_1R_2}{R_1+R_2}.
\end{aligned}
\end{equation}

Under the assumption (\ref{26}), we can neglect the terms containing $r^{-n}(n\geq 1)$ in (\ref{25}). Furthermore, by (\ref{A1}) and (\ref{B1}), we have
\begin{eqnarray}
\label{29} \left\{
   \begin{aligned}
 &\frac{\partial u_{r}}{\partial t}+(u \cdot \nabla)u_{r}=\Delta u_{r}-\frac{\partial p}{\partial r}+\bigg(\frac{\partial p}{\partial \theta}\bigg)_{0}\frac{2R_{1}R_{2}}{R_{1}+R_{2}}u_{\theta}
 \\&\frac{\partial u_{\theta}}{\partial t}+(u \cdot \nabla)u_{\theta}=\Delta u_{\theta}+\bigg(\frac{\partial  p}{\partial \theta}\bigg)_{0}\frac{R_{1}R_{2}}{R_{1}+R_{2}}u_{r},
 \\& \frac{\partial u_{z}}{\partial t}+(u \cdot \nabla)u_{z}=\Delta  u_{z}-\frac{\partial p}{\partial z},
 \\& \frac{\partial u_{r}}{\partial r}+\frac{\partial u_{z}}{\partial z}=0,
\end{aligned}
\right.
\end{eqnarray}
where $\Delta $ is defined as follows
\begin{equation*}
  \Delta =\frac{\partial^{2}}{\partial r^{2}}+\frac{\partial^{2}}{\partial z^{2}}.
\end{equation*}

Let $u_{r}=u'_{r}$, $u_{\theta}=\frac{1}{\sqrt{2}}u'_{\theta}$, $u_{z}=u'_{z}$, $p=p'$. Omitting the primes, we can rewrite the equations (\ref{29}) as follows
\begin{eqnarray}
\label{210} \left\{
   \begin{aligned}
 &\frac{\partial u_{r}}{\partial t}=\Delta  u_{r}-\frac{\partial p}{\partial r}+\bigg(\frac{\partial p}{\partial \theta}\bigg)_{0}\frac{\sqrt{2}R_{1}R_{2}}{R_{1}+R_{2}}u_{\theta}-(u \cdot \nabla)u_{r}
 \\&\frac{\partial u_{\theta}}{\partial t}=\Delta u_{\theta}+\bigg(\frac{\partial  p}{\partial \theta}\bigg)_{0}\frac{\sqrt{2}R_{1}R_{2}}{R_{1}+R_{2}}u_{r}-(u \cdot \nabla)u_{\theta}
 \\& \frac{\partial u_{z}}{\partial t}=\Delta u_{z}-\frac{\partial p}{\partial z}-(u \cdot \nabla)u_{z}
 \\& \frac{\partial u_{r}}{\partial r}+\frac{\partial u_{z}}{\partial z}=0,
\end{aligned}
\right.
\end{eqnarray}
\begin{remark}
It should be pointed out that 
the existence of global weak solution, the regularity of global weak solution and the existence of global attractor for the simplified model
(\ref{210}) under various boundary conditions has been obtained in our other submitted papers.
\end{remark}

\subsection{Abstract operator form}
For convenience, we denote $\lambda= \bigg(\frac{\partial  p}{\partial\theta}\bigg)_{0}\frac{\sqrt{2}R_{1}R_{2}}{R_{1}+R_{2}}$.
In this paper, we only consider the following free boundary conditions:
 \begin{eqnarray}
\label{211} \left\{
\begin{aligned}
&u_{r}=0,~u_{\theta}=0,~\frac{\partial u_{z}}{\partial r}=0,~~\mathrm{when}~ r=r_{1},~r=r_{2},
\\& \frac{\partial u_{r}}{\partial r}=\frac{\partial u_{\theta}}{\partial z}=0,~u_{z}=0,~~ \mathrm{when}~ z=0,L.
\end{aligned}
\right.
\end{eqnarray}
where $r_1=\frac{R_{1}}{l}$ and $r_2=\frac{R_{2}}{l}$ are non-dimensionless constants.

Let
 \begin{eqnarray}
\label{212}
\begin{aligned}
&H=\bigg\{u=((u_{r},u_{z}),u_{\theta})\in H^{2}(\Omega,\mathbb{R}^2)\times H^2(\Omega)~|~u ~ \mathrm{satifies}~(\ref{211})
\\ &\qquad\qquad\qquad\qquad\qquad\qquad\mathrm{and}~\frac{\partial u_{r}}{\partial r}+\frac{\partial u_{z}}{\partial z}=0\bigg\},
\\& H_1=\bigg\{u=((u_{r},u_{z}),u_{\theta})\in L^{2}(\Omega,\mathbb{R}^2)\times L^2(\Omega)~|~u ~ \mathrm{satifies}~(\ref{211})
\\ &\qquad\qquad\qquad\qquad\qquad\qquad\mathrm{and}~\frac{\partial u_{r}}{\partial r}+\frac{\partial u_{z}}{\partial z}=0\bigg\}.
\end{aligned}
\end{eqnarray}

Define the mappings $A,B,G:H\rightarrow H_{1}$ as follows
 \begin{eqnarray*}
\begin{aligned}
&Au=-\big(P(\Delta u_{r},\Delta u_{z}),\Delta u_{\theta}\big),
\\&Bu=\big(P( u_{\theta},0),u_{r}\big),
\\&Gu=\big(P((u\cdot\nabla)u_{r},(u\cdot\nabla)u_{z}),(u\cdot\nabla)u_{\theta}\big),
\end{aligned}
\end{eqnarray*}
where $P$ is the Leray projection.

Setting $L_{\lambda}=-A+\lambda B $, then the equation (\ref{210}) is equivalent to the following operator equation
 \begin{eqnarray}
\label{213} \left\{
\begin{aligned}
&\frac{du}{dt}=L_{\lambda}u+Gu,\quad t>0
\\& u_{0}=(u_{r0},u_{z0},u_{\theta0})
\end{aligned}
\right.
\end{eqnarray}
where $u_{0}$ is the initial value of (\ref{210}).

\section{Principle of exchange of stability}
Let the eigenvalue of operator $L_{\lambda}$ and the corresponding eigenvector be $\beta$ and $v=(v_{r},v_{\theta},v_{z}) $ respectively such that
\begin{equation}\label{2141}
  L_{\lambda}v=\beta v.
\end{equation}
Obviously, (\ref{2141}) is equivalent to
 \begin{eqnarray}
\label{215} \left\{
\begin{aligned}
&\Delta v_{r}+\lambda v_{\theta}-\frac{\partial p}{\partial r}=\beta v_{r},
\\&\Delta v_{\theta}+\lambda v_{r}=\beta v_{\theta},
\\& \Delta v_{z}-\frac{\partial p}{\partial z}=\beta v_{z},
\\&\frac{\partial v_{r}}{\partial r}+\frac{\partial v_{z}}{\partial z}=0.
\end{aligned}
\right.
\end{eqnarray}
Clearly, $v_{r},~v_{\theta}\in W_{1}$ and $v_{z}\in W_{2}$ respectively, where
\begin{eqnarray*}
\begin{aligned}
&W_{1}=\bigg\{\psi \in H^{2}(\Omega)~|~\psi=0,\mathrm{when}~r=r_{1},r_{2};~\frac{\partial \psi }{\partial z}=0,\mathrm{when}~z=0,L\bigg\},
\\&W_{2}=\bigg\{\psi \in H^{2}(\Omega)~|~\frac{\partial \psi }{\partial r}=0,\mathrm{when}~r=r_{1},r_{2};~\psi =0,\mathrm{when}~z=0,L\bigg\}.
\end{aligned}
\end{eqnarray*}

We know that
\begin{equation*}
  \Delta: W_{1}\rightarrow S_{1}
\end{equation*}
and
\begin{equation*}
  \Delta: W_{2}\rightarrow S_{2}
\end{equation*}
are isomorphic, where
\begin{eqnarray*}
\begin{aligned}
&S_{1}=\bigg\{\psi\in L^{2}(\Omega)~|~\psi=0,\mathrm{when}~r=r_{1},r_{2};~\frac{\partial \psi }{\partial z}=0,\mathrm{when}~z=0,L\bigg\},
\\&S_{2}=\bigg\{\psi \in L^{2}(\Omega)~|~\frac{\partial \psi }{\partial r}=0,\mathrm{when}~r=r_{1},r_{2};~\psi =0,\mathrm{when}~z=0,L\bigg\}.
\end{aligned}
\end{eqnarray*}
Noticing that the eigenvectors $e_{mn}$ and $\overline{e}_{mn}$ of $\Delta$  in $W_{1}$ and $W_{2}$  are the  orthogonal basis of $W_{1}$ and $W_{2}$ respectively. The eigenvalues $\lambda_{mn}$ and eigenvectors $e_{mn}$ of $\Delta$ in $W_{1}$ are as follows
\begin{equation*}
  \lambda_{mn}=-\pi^{2}\bigg(\frac{m^{2}}{L^{2}}+n^{2}\bigg), \quad e_{mn}=\cos\frac{m\pi z}{L}\sin n\pi(r-r_{1}).
\end{equation*}
The eigenvalues $\overline{\lambda}_{mn}$ and eigenvectors $\overline{e}_{mn}$ of $\Delta$ in $W_{2}$ are as follows
\begin{equation*}
  \overline{\lambda}_{mn}=\lambda_{mn}, \quad \overline{e}_{mn}=\sin\frac{m\pi z}{L}\cos n\pi(r-r_{1}),
\end{equation*}
where $m,n=1,2,\cdots$.

Hence, $v_{r},v_{\theta}$ are linear combinations of the eigenvectors $e_{mn}$ and $v_{z}$ is linear combination of the eigenvectors $\overline{e}_{mn}$, that is
\begin{eqnarray}\label{combination}
\begin{aligned}
&v_{r}=\sum_{m,n=1}^{\infty}a_{r}^{mn}e_{mn},
\\&v_{\theta}=\sum_{m,n=1}^{\infty}a_{\theta}^{mn}e_{mn},
\\&v_{z}=\sum_{m,n=1}^{\infty}a_{z}^{mn}\overline{e}_{mn}.
\end{aligned}
\end{eqnarray}

Substituting (\ref{combination}) into (\ref{215}) to get
\begin{eqnarray}\label{combination1}
\left\{
\begin{aligned}
&\sum_{m,n=1}^{\infty}(a_{r}^{mn}\lambda_{mn}e_{mn}+\lambda a_{\theta}^{mn}e_{mn})-\frac{\partial p}{\partial r}=\beta\sum_{m,n=1}^{\infty}a_{r}^{mn}e_{mn},
\\&\sum_{m,n=1}^{\infty}(a_{\theta}^{mn}\lambda_{mn}e_{mn}+\lambda a_{r}^{mn}e_{mn})=\beta\sum_{m,n=1}^{\infty}a_{\theta}^{mn}e_{mn},
\\&\sum_{m,n=1}^{\infty}a_{z}^{mn}\lambda_{mn}\overline{e}_{mn}-\frac{\partial p}{\partial z}=\beta\sum_{m,n=1}^{\infty}a_{z}^{mn}e_{mn},
\\&\sum_{m,n=1}^{\infty}\big(n\pi a_{r}^{mn}\cos\frac{m\pi z}{L}\cos n\pi(r-r_{1})
+\frac{m\pi}{L}a_{z}^{mn}\cos\frac{m\pi z}{L}\cos n\pi(r-r_{1})\big)
=0.
\end{aligned}
\right.
\end{eqnarray}
Eliminating $p$ of (\ref{combination1}) to get
\begin{eqnarray}\label{combination2}
\left\{
\begin{aligned}
&-\sum_{m,n=1}^{\infty}\big(a_{r}^{mn}\lambda_{mn}+\lambda a_{\theta}^{mn}\big)\frac{m\pi}{L}\sin\frac{m\pi z}{L}\sin n\pi(r-r_{1})
\\&\qquad\qquad\qquad\qquad +\sum_{m,n=1}^{\infty}a_{z}^{mn}\lambda_{mn}n\pi \sin\frac{m\pi z}{L}\sin n\pi(r-r_{1})
\\&\qquad\qquad\qquad\qquad=-\beta\sum_{m,n=1}^{\infty}\big(a_{r}^{mn}\frac{m\pi}{L}-n\pi a_{z}^{mn}\big)\sin\frac{m\pi}{L}z\sin n\pi(r-r_{1}),
\\&\sum_{m,n=1}^{\infty}(a_{\theta}^{mn}\lambda_{mn}e_{mn}+\lambda a_{r}^{mn}e_{mn})=\beta\sum_{m,n=1}^{\infty}a_{\theta}^{mn}e_{mn},
\\&\sum_{m,n=1}^{\infty}\big(n\pi a_{r}^{mn}\cos\frac{m\pi z}{L}\cos n\pi(r-r_{1})
+\frac{m\pi}{L}a_{z}^{mn}\cos\frac{m\pi z}{L}\cos n\pi(r-r_{1})\big)=0.
\end{aligned}
\right.
\end{eqnarray}

Furthermore, we obtain
 \begin{eqnarray}\label{216}\left\{
\begin{aligned}
&-\frac{m\pi}{L}\big(a_{r}^{mn}\lambda_{mn}+\lambda a_{\theta}^{mn}\big)+ n\pi a_{z}^{mn}\lambda_{mn}=-\beta \big(a_{r}^{mn}\frac{m\pi}{L}- a_{z}^{mn}n\pi\big),
\\&a_{\theta}^{mn}\lambda_{mn}+\lambda a_{r}^{mn}=\beta a_{\theta}^{mn},
\\&na_{r}^{mn}+\frac{m}{L}a_{z}^{mn}=0,
\end{aligned}
\right.
\end{eqnarray}
Then, we can deduce from (\ref{216}) that
\begin{equation} \label{217}      %开始数学环境
\left(                %左括号
  \begin{array}{cc}   %该矩阵一共2列，每一列都居中放置
    \lambda_{mn} & \frac{\lambda}{1+\frac{n^{2}L^{2}}{m^{2}}} \\  %第一行元素
    \lambda & \lambda_{mn}\\  %第二行元素
  \end{array}
\right)
\left(                %左括号
  \begin{array}{c}   %该矩阵一共1列，每一列都居中放置
   a_{r}^{mn} \\  %第一行元素
    a_{\theta}^{mn}\\  %第二行元素
  \end{array}
\right)
=\beta\left(                %左括号
  \begin{array}{c}   %该矩阵一共1列，每一列都居中放置
   a_{r}^{mn} \\  %第一行元素
    a_{\theta}^{mn}\\  %第二行元素
  \end{array}
\right).                %右括号
\end{equation}
Let
\begin{equation*}      %开始数学环境
A_{mn}=\left(                %左括号
  \begin{array}{cc}   %该矩阵一共2列，每一列都居中放置
    \lambda_{mn} & \frac{\lambda}{1+\frac{n^{2}L^{2}}{m^{2}}} \\  %第一行元素
    \lambda & \lambda_{mn}\\  %第二行元素
  \end{array}
\right),
~~X_{mn}=\left(                %左括号
  \begin{array}{c}   %该矩阵一共1列，每一列都居中放置
   a_{r}^{mn} \\  %第一行元素
    a_{\theta}^{mn}\\  %第二行元素
  \end{array}
\right),               %右括号
\end{equation*}
we rewrite (\ref{217}) as follows
\begin{equation}\label{218}
  A_{mn}X_{mn}=\beta X_{mn}.
\end{equation}
Obviously, the eigenvalue of (\ref{215}) is also the eigenvalue of  (\ref{218}).  It follows from (\ref{217}) that
 \begin{eqnarray}
\label{219}
&\begin{aligned}
\beta_{mn}^{1}=\lambda_{mn}+\frac{\lambda}{\sqrt{1+\frac{n^{2}L^{2}}{m^{2}}}},
\end{aligned}
\\& \label{220}
\begin{aligned}
\beta_{mn}^{2}=\lambda_{mn}-\frac{\lambda}{\sqrt{1+\frac{n^{2}L^{2}}{m^{2}}}}.
\end{aligned}
\end{eqnarray}
and the eigenvectors corresponding $\beta_{mn}^{i}(i=1,2)$ are given by
\begin{equation}\label{221}
  f_{mn}^{i}=\frac{2}{L}\bigg((-1)^{i+1}\frac{1}{\sqrt{1+\frac{n^{2}L^{2}}{m^{2}}}}e_{mn},e_{mn},
  (-1)^{i}\frac{nL}{m}\frac{1}{\sqrt{1+\frac{n^{2}L^{2}}{m^{2}}}}\overline{e}_{mn}\bigg).
\end{equation}

From (\ref{221}), it is easy to get the following properties of $f_{mn}^{i}$.
\begin{proposition}
Let $f_{mn}^{i}$ satisfy (\ref{221}), then
\begin{enumerate}
  \item
  $\{f_{mn}^{i}\}$ is an orthogonal  basis of $H$ and $H_1$, and
  \begin{eqnarray*}
  \left\{
   \begin{aligned}
   &\langle f_{mn}^{i},f_{sk}^{j}\rangle = 0, \quad(i,m,n)\neq (j,s,k),
   \\&\langle f_{mn}^{i},f_{sk}^{j}\rangle =1, \quad(i,m,n)= (j,s,k),
   \end{aligned}
   \right.
  \end{eqnarray*}
  \item
  \begin{equation*}
    \langle G(f_{mn}^{i},f_{sk}^{j}),f_{lh}^{p}\rangle = -\langle G(f_{mn}^{i},f_{lh}^{p}),f_{sk}^{j}\rangle,
  \end{equation*}
  where $ G(f_{mn}^{i},f_{sk}^{j})=(f_{mn}^{i}\cdot \nabla)f_{sk}^{j}$,
  \item
    \begin{equation*}
    \langle G(f_{mn}^{i},f_{sk}^{j}),f_{sk}^{j}\rangle = 0.
  \end{equation*}
\end{enumerate}

\end{proposition}

\subsection{First eigenvalue and critical number}
Obviously, it is easy to check that
\begin{equation*}
  \beta_{mn}^{2}=\lambda_{mn}-\frac{\lambda}{\sqrt{1+\frac{n^{2}L^{2}}{m^{2}}}}<0, ~~\mathrm{for}~~~~\lambda> 0.
\end{equation*}
Hence, the first eigenvalue of $L_{\lambda}$ is determined by the following formulation:
\begin{equation}\label{222}
 \beta_{1}=\max_{m,n\in N}\beta_{mn}^{1}=\max_{m\in N}\big(-\pi^{2}(1+\frac{m^{2}}{L^{2}})+\frac{\lambda}{\sqrt{1+\frac{L^{2}}{m^{2}}}} \big).
\end{equation}
(\ref{222}) implies
\begin{equation}\label{223}
  \beta_{m1}^{1}=0 ~~~\mathrm{if}~ ~~ \lambda=\lambda_{m}=\pi^{2}\sqrt{1+\frac{L^{2}}{m^{2}}}\big(1+\frac{m^{2}}{L^{2}}\big).
\end{equation}
Note that $\lambda_{m}$ is the function of $m$, which can obtain the minimum $\lambda_{0}$ when $ m=m_{0}\in N $.

Let $\lambda_0=\pi^{2}\sqrt{1+\frac{L^{2}}{m^{2}_0}}\big(1+\frac{m^{2}_0}{L^{2}}\big)$. Then there exists a positive $\delta>0$ such that $\lambda_{0}-\delta<\lambda<\lambda_{0}+\delta $, and $\beta_{1}=\beta_{m_{0}1}^{1}$,
  \begin{eqnarray}\label{224}
  \left\{
  \begin{aligned}
  &\beta_{m_{0}1}^{1}\left\{
   \begin{aligned}
   &<0, \quad \lambda_{0}-\delta<\lambda<\lambda_{0},
   \\&=0,\quad \lambda=\lambda_{0},
   \\&>0,\quad \lambda_{0}<\lambda<\lambda_{0}+\delta,
   \end{aligned}
   \right.
   \\&\beta_{mn}^{i}<0, \quad (i,m,n)\neq (1,m_{0},1).
   \end{aligned}
   \right.
  \end{eqnarray}
Let
\begin{equation*}
  E_{1}=\bigcup_{k\in N}\{x\in H|\big(L_{\lambda_{0}}-\beta_{m_{0}1}^{1}(\lambda_{0})\big)^{k}x=0\},
\end{equation*}
which  means that $\text{dim}E_{1}=1$.

\section{Main conclusions}
Based on Theorem 2.3.1\cite{Ma}, we can get the following conclusions.
\begin{theorem}
For system (\ref{213}), the following conclusion hold.
\begin{enumerate}
\item if $0\leq \lambda <\lambda_{0}$, it does not have bifurcation, and the only equilibrium point $u=0$ is globally  asymptotically stable;
 \item there exists a constant $\varepsilon >0$ such that $\lambda_{0} < \lambda <\lambda_{0}+\varepsilon $, system (\ref{213}) bifurcates from $(0,\lambda_{0})$ to an attractor $\Omega_{\lambda}$, and attractor $\Omega_{\lambda}$ attracts $H_{1}/\Gamma$, where $\Gamma$ is the stable manifold $(0,0)$, which has codimension $1$ in $H_{1}$.
\end{enumerate}
\end{theorem}

\begin{proof}
It is obvious that the eigenvalue $\beta^i_{mn}<0$ when $0<\lambda<\lambda_0$. Hence, the conclusion 1  naturally holds.

Note that for any $u\in H$,
\begin{equation*}
  \langle G(u),u\rangle=\int_\Omega(u\cdot u)udx=-\frac{1}{2}\bigg(\frac{\partial u_r}{
  \partial r}+\frac{\partial u_z}{
  \partial z}u^2\bigg)dx=0.
\end{equation*}
Based on the results in \cite{F}, system (\ref{213}) has a global attractor in $H$, that is, all invariant sets are uniformly bounded. Then we only need to prove that $0$ is the unique invariant set as $\lambda=\lambda_{0}$. We infer from (\ref{224}) that any invariant set $\Sigma$ of system (\ref{213} ) is only in $E_{1}$. Suppose $\Sigma$ has a non-zero element, for $u=(u_{r},u_{\theta},u_{z})\in \Sigma$, we know $u$ satisfies the following equations
\begin{eqnarray*}
\left\{
   \begin{aligned}
 &\frac{\partial u_{r}}{\partial t}=-\frac{\partial p}{\partial r}-(u \cdot \nabla)u_{r}
 \\&\frac{\partial u_{\theta}}{\partial t}=-(u \cdot \nabla)u_{\theta}
 \\& \frac{\partial u_{z}}{\partial t}=-\frac{\partial p}{\partial z}-(u \cdot \nabla)u_{z}
 \\& \frac{\partial}{\partial r}(u_{r})+\frac{\partial}{\partial z}(u_{z})=0,
\end{aligned}
\right.
\end{eqnarray*}
Clearly, $bu(bt)$ satisfies the above equations as well, $bu(bt)\in b\Sigma\subset E_1$, where $b$ is any real number. This contradicts to the boundaries of $\Sigma$. Then the invariant set of the system (\ref{213}) can only be $0$, and $0$ is a globally asymptotically stable attractor.
\end{proof}

\begin{theorem}
The set $\Omega_{\lambda}$ given in theorem 4.1, which contains two equilibrium point, that is $\Omega_{\lambda}=(\Omega_{1}^{\lambda},~\Omega_{2}^{\lambda})$. And there are following conclusion:
\begin{enumerate}
  \item $\Omega_{1}^{\lambda}=a\sqrt{\beta_{m_{0}1}^{1}}f_{m_{0}1}^{1}+o\big(\sqrt{\beta_{m_{0}1}^{1}}\big)$, \quad $\Omega_{2}^{\lambda}=-a\sqrt{\beta_{m_{0}1}^{1}}f_{m_{0}1}^{1}+o\big(\sqrt{\beta_{m_{0}1}^{1}}\big)$,
      \\where $a$ is a positive real number;
  \item $H_1$ is decomposed into two open sets $U_{1}$ and $U_{2}$:
  \begin{equation*}
    H_1=\overline{U}_{1}+\overline{U}_{2},\quad U_{1}\cap U_{2}=\phi,\quad 0\in \partial U_{1}\cap \partial  U_{2},
  \end{equation*}
such that $\Omega_{i}^{\lambda}\in U_{i} $(i=1,2), and
$$\lim\limits_{t\rightarrow \infty}\|u(t,\varphi)-\Omega_{i}^{\lambda}\|=0, \quad\text{for~ any}~ \varphi\in U_{i}.$$
\end{enumerate}
\end{theorem}

\begin{proof}
We only need to prove the conclusion 1.

It is clear that the steady state equation of (\ref{213}) is
\begin{equation}\label{225}
  L_{\lambda}u+Gu=0,\quad u\in H.
\end{equation}
Let the first feature space of $L_{\lambda}$ be $E$, and the orthogonal complement of $E$ be $E_{2}$, then
\begin{equation*}
  H=E_{1}\oplus E_{2},\quad u=u_{1}+u_{2}, \quad u_{i}\in E_{i},i=1,2.
\end{equation*}
Then (\ref{225}) can be rewritten as
\begin{align}
&\label{226}   L_{\lambda}u_{1}+P_{1}G(u_{1}+u_{2})=0,
\\&\label{227} L_{\lambda}u_{2}+P_{2}G(u_{1}+u_{2})=0.
\end{align}
where $P_{1}:H\rightarrow E_{1}$ and $P_{2}:H\rightarrow E_{2}$ are projections.

Thus, (\ref{226})-(\ref{227}) are equivalent to the equations
\begin{align}
\label{228}
   &\langle L_{\lambda}u_{1},f_{m_{0}1}^{1} \rangle+\langle G(u_{1}+u_{2}),f_{m_{0}1}^{1}\rangle =0,
\\\label{229}
  & \langle L_{\lambda}u_{2},f_{mn}^{j} \rangle+\langle G(u_{1}+u_{2}),f_{mn}^{j} \rangle =0,\quad (j,m,n)\neq (1,m_{0},1),
\end{align}
where $f_{m_{0}1}^{1}$ is the first eigenvector of $L_{\lambda}$.(\ref{228}) is the finite-dimensional bifurcation equations.

For any $u\in H$, $u$ can be expressed as
\begin{equation*}
  u=u_{1}+u_{2}=y_{m_{0}1}^{1}f_{m_{0}1}^{1}+\sum_{m>1,n,i\geq 1}y_{mn}^{i}f_{mn}^{i}, \quad  y_{m_{0}1}^{1},y_{mn}^{i}\in \mathbb{R}^{1}.
\end{equation*}
Then (\ref{228}) and (\ref{229}) are equivalent to the following two equations:
\begin{align*}
   &\beta_{m_{0}1}^{1}y_{m_{0}1}^{1}+\langle G(u_{1}+u_{2}),f_{m_{0}1}^{1}\rangle =0,
\\
  & \beta_{mn}^{i}y_{mn}^{i}+\langle G(u_{1}+u_{2}),f_{mn}^{i} \rangle =0,\quad (i,m,n)\neq (1,m_{0},1),
\end{align*}
By the proposition 1 and the implicit function theorem \cite{Ma0}, we can get
\begin{align}
\label{32}
   &\beta_{m_{0}1}^{1}y_{m_{0}1}^{1}-\sum_{m>1,n,i\geq 1}y_{m_{0}1}^{1}y_{mn}^{i}\langle G(f_{m_{0}1}^{1}),f_{mn}^{i}\rangle+O(|y_{mn}^{i}|^{2}) =0,
\\\label{33}
  & y_{mn}^{i}=-\frac{1}{\beta_{mn}^{i}}(y_{m_{0}1}^{1})^{2}\langle G(f_{m_{0}1}^{1}),f_{mn}^{i}\rangle+o(|y_{m_{0}1}^{1}|^{2}).
\end{align}
 Taking (\ref{33}) into (\ref{32}), we have
\begin{equation}\label{34}
  \beta_{m_{0}1}^{1}y_{m_{0}1}^{1}-2\gamma(y_{m_{0}1}^{1})^{3}+o(|y_{m_{0}1}^{1}|^{3})=0,
\end{equation}
where
\begin{equation*}
  \gamma=\sum_{m>1,n,i\geq 1}-\frac{1}{\beta_{mn}^{i}}\langle G(f_{m_{0}1}^{1}),f_{mn}^{i}\rangle^{2}>0.
\end{equation*}
From (\ref{224}), we can infer that (\ref{34}) bifurcates two singular points from $(0,\lambda_{0})$ when $\lambda>\lambda_{0}$:
\begin{equation*}
  (y_{m_{0}1}^{1})_{1,2}=\pm\sqrt{\frac{\beta_{m_{0}1}^{1}}{\gamma}}+o(|\beta_{m_{0}1}^{1}|^{\frac{1}{2}}).
\end{equation*}
Then
\begin{equation}\label{35}
 \Omega_{1}^{\lambda}=a\sqrt{\beta_{m_{0}1}^{1}}f_{m_{0}1}^{1}
 +o(|\beta_{m_{0}1}^{1}|^{\frac{1}{2}}),~
 \Omega_{2}^{\lambda}=-a\sqrt{\beta_{m_{0}1}^{1}}f_{m_{0}1}^{1}+o(|\beta_{m_{0}1}^{1}|^{\frac{1}{2}}),
\end{equation}
\\where $a=\frac{1}{\sqrt{\gamma}}>0$.
\end{proof}

\textbf{Physical explanation of theorem 4.2}. The fact is that $R_{1},R_{2},\rho$ and $\nu$ are fixed, and we can only adjust the value of $\bigg(\frac{\partial p}{\partial \theta}\bigg)_{0}$. Then (\ref{224}) and the expression of $\lambda$ imply that the basic flow becomes unstable, and the movement of the flow in $r$ and $z$ directions will occur  if $\bigg(\frac{\partial p}{\partial \theta}\bigg)_{0}$ exceeds a certain critical value. The structure of flow in $r$ and $z$ directions is shown in Figure 1, and the number of vortexes is determined by $L$ and $m_{0}$.

\begin{figure}[H]
  \centering
  \includegraphics[width=1.0\textwidth]{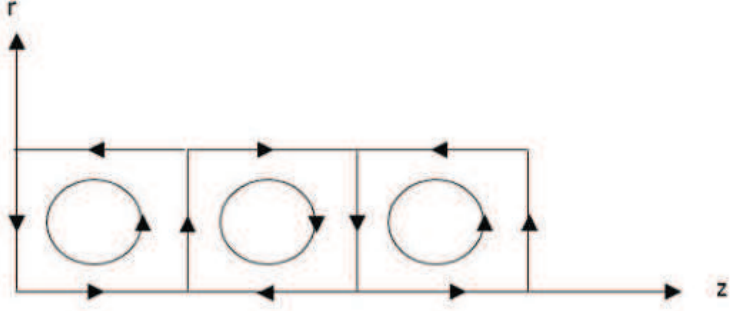}\\
  \caption{Structure}\label{1}
\end{figure}

\end{document}